%% file: C2AdamsTMF.tex
\documentclass{article}
\input{preamble.tex}

\title{Stable Adams operations on $RO(C_2)$-graded homotopy groups}

\begin{document}

\maketitle 
\begin{abstract}
    The $C_2$-spectrum of Atiyah's Real $K$-theory is denoted by $\KR$ and the $C_2$-spectrum of topological modular forms of level structure $\Gamma_1(3)$ by $\CTMF_1(3)$.
    In this short note we compute the $C_2$-equivariant stable Adams operations on the $RO(C_2)$-graded homotopy groups of $\KR$ and $\CTMF_1(3)$. 
\end{abstract}

\section{Introduction} \label{sec:introduction}
In the study of the connection between elliptic curves and cohomology theories, Hopkins \cite{Hopkins:TopologicalModularForms} constructed the universal elliptic cohomology theory $\TMF$ of topological modular forms.
This is the spectrum corresponding to the moduli stack $\Me$ of elliptic curves.
Analogous to the situation in classical modular forms, there exist spectra of topological modular forms of certain level structures.
The moduli stack $\topos M_0(n)$, classifying elliptic curves with a chosen cyclic subgroup of order $n$, yields the spectrum $\TMF_0(n)$, and the moduli stack $\topos M_1(n)$, which classifies elliptic curves with a chosen point of exact order $n$, has associated spectrum $\TMF_1(n)$.

Moreover, Adams operations on $K$-theory proved to be important tools in the past.
Two of the major uses were Adams and Atiyah's \cite{AdamsAtiyah:HopfOne} proof of the Hopf invariant one problem using unstable Adams operations and Quillen's \cite{Quillen:CohomologyKTheoryGL} computation of the algebraic $K$-theory of finite fields using 
stable Adams operations.
As topological modular forms are to be viewed as a higher chromatic analogue of topological $K$-theory, it is natural to ask for the existence of stable Adams operations 
on the different versions of topological modular forms.
These were constructed by Davies \cite{Davies:Hecke,Davies:ConstructingCalculatingAO}.

Recall that the $C_2$-action on a complex vector bundle is given by complex conjugation, which carries over to the spectrum $\KU$ of complex $K$-theory.
The $C_2$-action on the spectrum $\TMF_1(3)$ is induced by an underlying action as well:
there is a natural $(\Z/n)^\times$-action on $\topos M_1(n)$ given by sending the point $x$ of order $n$ to $[k]x$ for $k \in (\Z/n)^\times$.
It follows that both, $\KU$ and $\TMF_1(3)$, can be made genuinely equivariant.
We will denote the $C_2$-spectra by $\KR$ and $\CTMF_1(3)$.
There is even more to say: in both cases the endomorphism $[-1]$ induces the endomorphism $\psi^{-1}$ of spectra.
If $\psi^k$ is the endomorphism defined by e.g., \cite[Definition 2.1]{Davies:Hecke}, \cite[Definition 6.16]{Davies:OnLuriesTheorem}, we even have coherently that $\psi^{-1} \circ \psi^k \simeq \psi^k \circ \psi^{-1}$ and $\psi^{k} \circ \psi^l \simeq \psi^{kl}$ making the $\psi^k$ $C_2$-equivariant.
These endomorphisms $\psi^k$ are then called \emph{stable Adams operations}. 
See also \cref{sec:recollection} for more details.

The goal of this paper is to compute the $C_2$-equivariant stable Adams operations on the representation graded homotopy groups of these $C_2$-spectra.
Since $K$-theory and topological modular forms are so closely related it is not surprising that the result and the strategy for the computations are essentially the same.
\begin{introthm}[\cref{thm:KR:main-thm}]
    Let $k \in \Z$ and $x \in \pi_{a+b\sigma}^{C_2} \KR[\frac{1}{k}]$, where $\sigma$ is the sign representation of $C_2$.
    Then in $\pi_{a+b\sigma}^{C_2} \KR[\frac{1}{k}]$ \[\psi^k(x) = \begin{cases}
        k^{\frac{a + b}{2}}x & \textup{if } a+b \textup{ is even}\\
        x & \textup{if } a+b \textup{ is odd.}
    \end{cases}\]
    In particular, if $x$ is torsion (necessarily $2$-torsion), then $\psi^k(x)=x$.
\end{introthm}
\begin{introthm}[\cref{thm:TMF:main-thm}]
    Let $k \in \Z$ and $x \in \pi_{a+b\sigma}^{C_2} \CTMF_1(3)[\frac{1}{k}]$, where $\sigma$ is the sign representation of $C_2$.
    Then in $\pi_{a+b\sigma}^{C_2}  \CTMF_1(3)[\frac{1}{k}]$ \[\psi^k(x) 
    \begin{cases}
        k^{\frac{a + b}{2}}x & \textup{if } a+b \textup{ is even}\\
        x & \textup{if } a+b \textup{ is odd.}
    \end{cases}\]
    In particular, if $x$ is torsion (necessarily $2$-torsion), then $\psi^k(x)=x$.
\end{introthm}
The computations build on the fact that stable Adams operations are maps of $\E_\infty$-rings, and it therefore suffices to compute them on the generators of the homotopy groups.
There are three different kinds of generator, and for each we use a different technique.
If the generator is
\begin{itemize}
    \item coming from the sphere, we use $\S$-linearity.
    \item coming from the underlying nonequivariant spectrum, use the underlying computations.
    \item genuinely equivariant, use the homotopy fixed point spectra.
\end{itemize}

In \cref{sec:recollection}, we will review relevant results about the computations of the $RO(C_2)$-graded homotopy groups of $\KR$ and $\CTMF_1(3)$
as well as the construction of stable Adams operations.
The main theorems are proven in \cref{sec:main theorems}.
We will first deal with $K$-theory and then analogously handle topological modular forms.

\subsection*{Acknowledgments} 
The results of this article are part of my Masters' thesis at the University of Bonn.
I thank my advisor Jack Davies for suggesting the topic, answering countless questions and giving invaluable advice on a draft of this paper.
I am especially grateful to John Greenlees for sharing his LaTeX code for \cref{figure:preleiminaries:homotopyKR,figure:preleiminaries:homotopyTMF13}.
I also thank Emma Brink, Lennart Meier and Lucas Piessevaux for many helpful discussions about $\TMF$, and Tom Bachmann and Lucas Piessevaux for reading drafts of this article at various stages.
Many thanks to the anonymous referees for helpful suggestions that improved this article.
The author acknowledges support by the Deutsche Forschungsgemeinschaft
(DFG, German Research Foundation) through the Collaborative Research Centre TRR 326 
\textit{Geometry and Arithmetic of Uniformized Structures}, project number 444845124.

\section{Recollections} \label{sec:recollection}
The purpose of this section is to recall the necessary information about the homotopy groups of the $C_2$-spectra of Real $K$-theory and topological modular forms with $\Gamma_1(3)$ level structure.
We will only give the parts needed for the computations in \cref{sec:main theorems} and references for more details.
We also quickly recall the existence of stable Adams operations on the relevant spectra.

Let $\sigma$ be the sign representation of $C_2$.
Recall that the representation ring is $RO(C_2) \simeq \Z[\sigma]/(\sigma^2 - 1)$.
The $RO(C_2)$-graded homotopy groups of interest will have two generators which have a truly equivariant nature.
The first one is the Euler class of $\sigma$, which is denoted as $a_\sigma$, cf.\ \cite[Definition 3.11]{HillHopkinsRavenel:KervaireInvariantNonExistence}.
Note that this class $a_\sigma$ corresponds to the inclusion $S^0 \to S^\sigma$ and lives in degree $-\sigma$.
The second class is the Thom-like class $u_{2\sigma}$ living in degree $2(1-\sigma)$, see \cite[Definition 3.12]{HillHopkinsRavenel:KervaireInvariantNonExistence}.

\subsection*{Real $K$-theory} \label{sec:recollection:K}
Let us denote the genuine $C_2$-spectrum of Real $K$-theory of \cite{Atiyah:KtheoryReality} by $\KR$.
Moreover, we denote the Bott element by $u \in \pi_2 \KU$.
This has an equivariant lift $\overline{u} \in \pi^{C_2}_{1+ \sigma} \KR$, i.e., the restriction morphism $\text{res}^{C_2}_e \colon \pi^{C_2}_{1+\sigma} \KR \to \pi^e_{2} \KU$ sends $\overline{u}$ to $u$.
Computations of e.g., \cite{Dugger:AtHiSpSeqforKR,BrunerGreenlees:ConnectiveRealK,Greenlees:4reality} show, that the $RO(C_2)$-graded homotopy groups $\pi_\bigstar^{C_2} \KR$ are generated by the elements:
\begin{equation}\label{eq:generators:KR}
   a_\sigma \textup{, } \hspace{10pt} u^{\pm 2}_{2\sigma} \textup{, } \hspace{10pt} \overline{u} \hspace{10pt} \textup{ and  } \hspace{10pt} v_0(m) \coloneqq 2u^m_{2\sigma} \textup{ for } m \in \Z
\end{equation}
such that $2v_0(2)= 4 u^2_{2\sigma}$.
Note that this notation is from the homotopy fixed point spectral sequence, and is not meant to imply that this is a product on $\pi_\bigstar$.
The generators of \cref{eq:generators:KR} are subject to several relations which we do not include here, as they are not relevant for our computations.
Instead, we depict the homotopy groups of $\kR$, which give these of $\KR$ by inverting $\overline{u}$, in \cref{figure:preleiminaries:homotopyKR}.
This figure was taken from \cite{Greenlees:4reality}, but the notation was adapted to ours.
\begin{figure}[h]
    \begin{center}
        \begin{tikzpicture}[scale =0.8]
            \small
                \clip (-5, -5.5) rectangle (5, 6);
                \draw[step=0.5, gray!20, very thin] (-9,-9) grid (9, 9);

                \draw[blue, ->, thick](1.6,0)--(4.8, 0);
                \draw (2.9, 0) node[anchor=north]{$\Z\cdot 1$};

                \draw[blue, ->, thick](0,1.65)--(0, 5.4);
                \draw (0,4.5) node[anchor=east]{$\Z\cdot \sigma$};

                \draw (-2.5, -3.5) node[anchor=east, draw=orange]{\Large{$\pi^{C_2}_\bigstar \kR$}};

                \draw [->] (0,0)-- (8, 8);
                \node at (0,0)  [shape = rectangle, draw]{};
                \draw (0,0) node[anchor=east]{1};

                \draw [->](1,-1)-- (9, 7);
                \node at (1,-1) [shape=circle, draw] {};
                \draw (1,-1) node[anchor=north]{$v_0(1)$};
                \draw (-1,1) node[anchor=east]{$v_0(-1)$};

                \draw[purple] (0,0) -- (0,-8);

                \draw[red] (0,-0.5)-- (8, 7.5);
                \draw[red] (0,-1)-- (8, 7);
                \foreach \y in {1, 2, 3,4,5,6,7,8,9,10,11,12,13,14,15}
                \draw (0,-\y/2) node[anchor=east] {$a^{\y}_\sigma$};
                \foreach \y in {1,2,3,4,5,6,7,8,9,10,11,12,13,14,15}
                \node at (0,-\y/2) [fill=red, inner sep=1pt, shape=circle, draw] {};


                \draw [->] (2,-2)-- (8, 4);
                \node at (2,-2)  [shape = rectangle, draw]{};
                \draw (2,-2) node[anchor=east]{$u^2_{2\sigma}$};
                \draw (-2,2) node[anchor=east]{$2u^{-2}_{2\sigma}$};

                \draw [->](3,-3)-- (9, 3);
                \node at (3,-3) [shape=circle, draw] {};
                \draw (3,-3) node[anchor=north]{$v_0(3)$};
                \draw (-3,3) node[anchor=east]{$v_0(-3)$};

                \draw[purple] (2,-2) -- (2,-8);

                \draw[red] (2,-2.5)-- (7, 2.5);
                \draw[red] (2,-3)-- (7, 2);
                \foreach \y in {1,2,3,4,5,6,7,8,9,10,11,12,13,14,15}
                \node at (2,-2-\y/2) [fill=red, inner sep=1pt, shape=circle, draw] {};


                \draw [->] (4,-4)-- (8, 0);
                \node at (4,-4)  [shape = rectangle, draw]{};
                \draw (4,-4) node[anchor=east]{$u^4_{2\sigma}$};
                \draw (-4,4) node[anchor=east]{$2u^{-4}_{2\sigma}$};

                \node at (5,-5) [shape=circle, draw] {};

                \draw[purple] (4,-4) -- (4,-8);

                \draw[red] (4,-4.5)-- (7, -1.5);
                \draw[red] (4,-5)-- (7, -2);
                \foreach \y in {1,2,3,4,5,6,7,8,9,10,11,12,13,14,15}
                \node at (4,-4-\y/2) [fill=red, inner sep=1pt, shape=circle, draw] {};



                \draw (-2,2)-- (5, 9);
                \draw (4,8) node[anchor=east]{$(2, \overline{v}_1)P$};
                \node at (-2,2) [shape=circle, draw] {};
                \node at (-1.5,2.5) [shape=rectangle, draw] {};

                \draw (-1,1)-- (7, 9);
                \node at (-1,1) [shape=circle, draw] {};

                \foreach \y in {0, 1,2,3,4,5,6,7,8}
                \node at (-2.5,2.5+\y/2) [fill=red, inner sep=1pt, shape=circle, draw]
                {};
                \draw[red] (-2.5, 2.5) --(-2.5, 8);

                \node at (-1.5,1.5) [fill=red, inner sep=1pt, shape=circle, draw] {};
                \draw[red] (-1.5,1.5)-- (5, 8);
                \node at (-1.5,2.0) [fill=red, inner sep=1pt, shape=circle, draw] {};
                \draw[red] (-1.5,2.0)-- (5, 8.5);

                \draw (-4,4)-- (5, 13);
                \node at (-4,4) [shape=circle, draw] {};
                \node at (-3.5,4.5) [shape=rectangle, draw] {};

                \draw (-3,3)-- (7, 13);
                \node at (-3,3) [shape=circle, draw] {};

                \foreach \y in {0, 1,2,3,4,5,6,7,8}
                \node at (-4.5,4.5+\y/2) [fill=red, inner sep=1pt, shape=circle, draw]
                {};
                \draw[red] (-4.5, 4.5) --(-4.5, 8);

                \node at (-3.5,4) [fill=red, inner sep=1pt, shape=circle, draw] {};
                \draw[red] (-3.5,4)-- (5, 12.5);
                \node at (-3.5,3.5) [fill=red, inner sep=1pt, shape=circle, draw] {};
                \draw[red] (-3.5,3.5)-- (5, 12);

                \node at (0.5,0.5)  [shape = rectangle, draw]{};
                \draw (0.5,0.5) node[anchor=east]{$\overline{u}$};

        \end{tikzpicture}
        \caption{Black circles, squares and lines represent copies of $\Z$, red dots and lines represent copies of $\finfld{2}$. For more details we refer to \cite[§3]{Greenlees:4reality}.}
        \label{figure:preleiminaries:homotopyKR}
    \end{center}
\end{figure}

Nonequivariantly, stable Adams operations on the spectra $\KU$ of complex and $\KO$ of real topological $K$-theory were first introduced and completely computed by Adams \cite{Adams:VectorFieldsonSpheres}.
By \cite[Example 0.0.4]{EllipticII} the $C_2$-action on $\KU$ is given by $[-1] \colon \hat{\G}_m \to \hat{\G}_m$ acting on the canonical orientation of $\hat{\G}_m$ over $\KU$.
Hence, by the discussion in \cite[§6.4]{Davies:OnLuriesTheorem}, this yields $C_2$-equivariant stable Adams operations $\psi^k$ on $\KR[\frac{1}{k}]$ for $k \in \Z$, see also \cite[Remark 6.7]{DaviesLinskens:Global-Tate-K}.

\subsection*{Topological modular forms of level structure $\Gamma_1(3)$}
We denote the Borel $C_2$-spectrum of \cite[Theorem 2.4]{HillMeier:C2Tmf13} of topological modular form with $\Gamma_1(3)$ level structure by $\CTMF_1(3)$.
Its underlying nonequivariant spectrum is denoted by $\TMF_1(3)$ with homotopy groups $\pi_\ast \TMF_1(3) \cong \Z[\frac{1}{3}][a_1,a_3, \Delta^{-1}]$ \cite[Corollary 3.3]{MahowaldRezk:TMFofLevel3}.
Here $|a_1| =2$, $|a_3| = 6$ and $\Delta =a_3^3(a_1^3-27a_3)$.
The equivariant lifts are denoted by $\overline{a}_1$ and $\overline{a}_3$ and lie in degree $1+ \sigma$ and $3 + 3\sigma$, respectively.
Moreover, Hill and Meier \cite[§4.3]{HillMeier:C2Tmf13} show that as $C_2$-spectra $\Ctmf_1(3)[\overline{\Delta}^{-1}] \cong \CTMF_1(3)$, where
$\overline{\Delta} = \overline{a}_3^3(\overline{a}_1^3 - 27 \overline{a}_3)$ is the equivariant lift of $\Delta$.

Combining \cite[Theorem 4.15]{HillMeier:C2Tmf13} and \cite[§13.B]{GreenleesMeier:Gorenstein} one obtains that $\pi^{C_2}_\bigstar \CTMF_1(3)$ is generated by the following classes:
\begin{equation} \label{eq:generators:TMF}
   a_\sigma \textup{, } \hspace{10pt} \overline{a}_1 \textup{, } \hspace{10pt}  \overline{a}_3 \textup{, } \hspace{10pt} u^4_{2\sigma} \textup{, } \hspace{10pt} \overline{a}_1(1) \coloneqq \overline{a}_1u^{2}_{2\sigma}  \textup{ and } \hspace{10pt} v_0(m) \coloneqq 2 u^m_{2\sigma} \textup{ for }  m \in \Z.
\end{equation}
The same remarks as in the $K$-theory case hold.
A picture of the homotopy groups of $\Ctmf_1(3)$, which give these of $\CTMF_1(3)$ by inverting $\overline{\Delta}$, is \cref{figure:preleiminaries:homotopyTMF13}, which is taken from \cite[§13.B]{GreenleesMeier:Gorenstein}.
\begin{figure}[]
    \begin{center}
        \begin{tikzpicture}[scale =1.6]
            \small
            \clip (-3.5, -3.5) rectangle (3.5, 3.5);
            \draw[step=0.25, gray!20, very thin] (-7,-6) grid (9, 9);
            \draw (-1.25,-1.0) node[anchor=east, draw=orange]{\Large{$\pi^{C_2}_{\bigstar}\Ctmf_1(3)$}};
            \foreach \n in {-7,-6,-5, -4, -3,-2,-1,0,1,2,3,4}
            \draw [->] (\n,-\n)-- (\n +8,8-\n);
            \foreach \n in {0,1,2,3,4}
            \node at (\n,-\n)  [shape = rectangle, draw]{};
            \draw (0,0) node[anchor=east]{$1$};
            \draw (1,-1) node[anchor=east]{$u^{4}_{2\sigma}$};
            \draw (2,-2) node[anchor=east]{$u^{8}_{2\sigma}$};
            \draw (3,-3) node[anchor=east]{$u^{12}_{2\sigma}$};
            \draw (4,-4) node[anchor=east]{$u^{16}_{2\sigma}$};

            \foreach \n in {-7,-6,-5, -4, -3,-2,-1}
            \node at (\n,-\n)  [shape = circle, draw]{};
            \draw (-1,1) node[anchor=east]{$2\cdot u_{2\sigma}^{-4}$};
            \draw (-2,2) node[anchor=east]{$2\cdot u_{2\sigma}^{-8}$};
            \draw (-3,3) node[anchor=east]{$2\cdot u_{2\sigma}^{-12}$};

            \foreach \n in {-7,-6,-5, -4, -3,-2,-1,0,1,2,3,4}
            \draw [->] (\n+0.5,-\n-0.5)-- (\n  +8.5,7.5-\n);
            \foreach \n in {-7,-6,-5, -4, -3,-2,-1,0,1,2,3,4}
            \node at (\n+0.5,-\n-0.5) [fill=black, inner sep=1pt, shape=circle, draw] {};
            \foreach \n in {-7,-6,-5, -4, -3,-2,-1,0,1,2,3,4}
            \draw [->] (\n+0.28,-\n-0.28)-- (\n  +8.28,7.72-\n);
            \foreach \n in {-7,-6,-5, -4, -3,-2,-1,0,1,2,3,4}
            \node at (\n+0.28,-\n-0.28) [inner sep=1pt, shape=circle, draw] {};
            \foreach \n in {-7,-6,-5, -4, -3,-2,-1,0,1,2,3,4}
            \draw [->] (\n+0.75,-\n-0.75)-- (\n  +8.75,7.25-\n);
            \foreach \n in {-7,-6,-5, -4, -3,-2,-1,0,1,2,3,4}
            \node at (\n+0.75,-\n-0.75) [inner sep=1pt, shape=circle, draw] {};

            \foreach \n in {0,...,4}
            \draw [red] (\n,-\n-0.125)-- (\n +8,7.875-\n);
            \foreach \n in {0,...,4}
            \draw [red] (\n,-\n-0.25)-- (\n +8,7.75-\n);
            \foreach \n in {0,...,4}
            \foreach \y in {1,...,40}
            \node at (\n,-\n-\y/8) [fill=red, inner sep=1pt, shape=circle,
            draw]{};

            \foreach \n in {-7,...,-1}
            \foreach \y in {1,...,40}
            \node at (\n-0.125,-\n+\y/8) [fill=red, inner sep=1pt, shape=circle, draw]{};

            \foreach \n in {-7,-6,-5, -4, -3,-2,-1}
            \draw [red] (\n+0.125,-\n)-- (\n +8.125,8-\n);
            \foreach \n in {-7,-6,-5, -4, -3,-2,-1}
            \node at (\n+0.125,-\n) [fill=red, inner sep=1pt, shape=circle, draw] {};
            \foreach \n in {-7,-6,-5, -4, -3,-2,-1}
            \draw [red] (\n+0.125, -\n-0.125)-- (\n +8.125,7.875-\n);
            \foreach \n in {-7,-6,-5, -4, -3,-2,-1}
            \node at (\n+0.125,-\n-0.125) [fill=red, inner sep=1pt, shape=circle, draw] {};
            \foreach \n in {0,1,2,3,4}
            \draw [green] (\n,-\n-0.375)-- (\n +8,7.625-\n);
            \foreach \n in {0,1,2,3,4}
            \node at (\n,-\n-0.375) [fill=red, inner sep=1pt, shape=circle, draw] {};
            \foreach \n in {0,1,2,3,4}
            \draw [green] (\n,-\n-0.5)-- (\n +8,7.5-\n);
            \foreach \n in {0,1,2,3,4}
            \node at (\n,-\n-0.5) [fill=red, inner sep=1pt, shape=circle, draw] {};
            \foreach \n in {0,1,2,3,4}
            \draw [green] (\n,-\n-0.625)-- (\n +8,7.375-\n);
            \foreach \n in {0,1,2,3,4}
            \node at (\n,-\n-0.625) [fill=red, inner sep=1pt, shape=circle, draw] {};
            \foreach \n in {0,1,2,3,4}
            \draw [green] (\n,-\n-0.75)-- (\n +8,7.25-\n);
            \foreach \n in {0,1,2,3,4}
            \node at (\n,-\n-0.75) [fill=red, inner sep=1pt, shape=circle, draw] {};

            \foreach \n in {-7,-6,-5, -4, -3,-2,-1}
            \draw [green] (\n+0.375, -\n)-- (\n +8.375, 8-\n);
            \foreach \n in {-7,-6,-5, -4, -3,-2,-1}
            \node at (\n+0.375, -\n) [fill=red, inner sep=1pt, shape=circle, draw] {};
            \foreach \n in {-7,-6,-5, -4, -3,-2,-1}
            \draw [green] (\n+0.375,-\n-0.125)-- (\n +8.375,7.875-\n);
            \foreach \n in {-7,-6,-5, -4, -3,-2,-1}
            \node at (\n+0.375,-\n-0.125) [fill=red, inner sep=1pt, shape=circle, draw] {};
            \foreach \n in {-7,-6,-5, -4, -3,-2,-1}
            \draw [green] (\n+0.375,-\n-0.25)-- (\n +8.375,7.75-\n);
            \foreach \n in {-7,-6,-5, -4, -3,-2,-1}
            \node at (\n+0.375,-\n-0.25) [fill=red, inner sep=1pt, shape=circle, draw] {};
            \foreach \n in {-7,-6,-5, -4, -3,-2,-1}
            \draw [green] (\n+0.375,-\n-0.375)-- (\n +8.375,7.625-\n);
            \foreach \n in {-7,-6,-5, -4, -3,-2,-1}
            \node at (\n+0.375,-\n-0.375) [fill=red, inner sep=1pt, shape=circle, draw] {};

            \foreach \n in {0,1,2,3,4}
            \draw [red] (\n+0.625,-\n-0.5)-- (\n +8.625,7.5-\n);
            \foreach \n in {0,1,2,3,4}
            \node at (\n+0.625,-\n-0.5) [fill=red, inner sep=1pt, shape=circle, draw] {};
            \foreach \n in {0,1,2,3,4}
            \draw [red] (\n+0.625,-\n-0.625)-- (\n +8.625,7.375-\n);
            \foreach \n in {0,1,2,3,4}
            \node at (\n+0.625,-\n-0.625) [fill=red, inner sep=1pt, shape=circle, draw] {};

            \foreach \n in {-7,-6,-5, -4, -3,-2,-1}
            \draw [red] (\n+0.625,-\n-0.5)-- (\n +8.375,7.5-\n);
            \foreach \n in {-7,-6,-5, -4, -3,-2,-1}
            \node at (\n+0.625,-\n-0.5) [fill=red, inner sep=1pt, shape=circle, draw] {};
            \foreach \n in {-7,-6,-5, -4, -3,-2,-1}
            \draw [red] (\n+0.625, -\n-0.625)-- (\n +8.375,7.375-\n);
            \foreach \n in {-7,-6,-5, -4, -3,-2,-1}
            \node at (\n+0.625,-\n-0.625) [fill=red, inner sep=1pt, shape=circle, draw] {};

            \node at (0.125,0.125)  [shape = rectangle, draw]{};
            \draw (0.125,0.125) node[anchor=south]{$\overline{a}_1$};
            \node at (0.375,0.375)  [shape = rectangle, draw]{};
            \draw (0.375,0.375) node[anchor=south]{$\overline{a}_3$};

        \end{tikzpicture}
        \caption{Black circles, squares and lines represent copies of $\Z$; red dots represent copies of $\finfld{2}$, red lines mean a copy of $\finfld{2}[\overline{a}_1, \overline{a}_3]$ and green lines mean a copy of $\finfld{2}[\overline{a}_3]$. For more details we refer to \cite[§13.B]{GreenleesMeier:Gorenstein}. The axes are the same as in \cref{figure:preleiminaries:homotopyKR}}
        \label{figure:preleiminaries:homotopyTMF13}
    \end{center}
\end{figure}

The $C_2$-spectrum $\CTMF_1(3)$ has $C_2$-equivariant stable Adams operations by \cite[Theorem C and Proposition 2.15]{Davies:Hecke} and the remarks directly after these.
Let us make this more precise.
By \cite[Theorem A]{Davies:Hecke}, there is a functor $\Ot \colon \op{\Isog} \to \CAlg{\Sp}$.
Here, $\Isog$ is the $2$-category whose objects are stacks $\topos X$ equipped with an étale map to the moduli stack of elliptic curves $\Me$, which classifies an elliptic curve $E$ over $\topos X$, and morphisms are pairs $(f,\varphi)$ of a map of stacks $f\colon \topos X' \to \topos X$ and an isogeny of elliptic curves $\varphi \colon E' \to f^\ast E$ of invertible degree. 
The $\E_\infty$-ring $\TMF_1(3)$ is the image of the moduli stack $\MT$ of elliptic curves with $\Gamma_1(3)$-level structure \cite[§2]{MahowaldRezk:TMFofLevel3} with $C_2$-action given by the automorphism $(id, [-1])$ of $(\MT, E_1(3))$, where $E_1(3)$ is the universal elliptic curve living over $\MT$.
Let us denote this $C_2$-action on $\TMF_1(3)$ by $\Phi$.
By \cite[Diagram (2.2)]{Davies:Hecke}, we define $\psi^k$ on $\TMF_1(3)[\tfrac{1}{k}]$ using the endomorphism $(id, [k])$ of $(\MT, E_1(3))$ and applying $\Ot$. 
For the following we use the same ideas as in the proof of \cite[Proposition 2.15]{Davies:Hecke}.
By the functoriality of $\Ot$ in isogenies of invertible degree, we obtain the following chain of natural homotopies:
\[\psi^k \circ \Phi = \Ot(\id{}, [k]) \circ \Ot(\id{}, [-1]) \simeq \Ot(\id{}, [k] \circ [-1]).\]
As $[k]$ is a homomorphism of elliptic curves, i.e., a homomorphism of abelian group objects in schemes, it commutes with taking inverses.
Therefore, \[\Ot(\id{}, [k] \circ [-1]) \simeq \Ot(\id{}, [-1] \circ [k]) \simeq \Phi \circ \psi^k,\]
and the Adams operations $\psi^k$ on $\TMF_1(3)[\tfrac{1}{k}]$ are $C_2$-equivariant in $\CAlg{\Sp}$.

\begin{rmk}
    Using Hill and Meier's theory of Real Landweber exactness \cite[Theorem 3.6]{HillMeier:C2Tmf13} one could obtain $C_2$-equivariant stable Adams operations on the $C_2$-homology theories $\KR$ and $\CTMF_1(3)$.
    But it is not clear to the author if this construction lifts the Adams operations to the $C_2$-spectra, thus we chose the constructions above.
\end{rmk}

\section{Computations of Adams operations} \label{sec:main theorems}
In this section, we compute the stable Adams operations on the $RO(C_2)$-graded homotopy groups of $\KR$ and $\CTMF_1(3)$.

\subsection*{Real $K$-theory}
Here we compute the effect of stable Adams operations on the $RO(C_2)$-graded homotopy groups of the $C_2$-spectrum $\KR$.
There is a second $C_2$-spectrum related to topological complex $K$-theory, which is denoted by $\KU_{C_2}$ and defined by Segal \cite{Segal:EquivariantKTheory}.
Note that $\KR \neq \KU_{C_2}$, but there is a map $\KU_{C_2} \to \KR$.
Balderrama \cite[Lemma 2.2.2]{Balderrama:KTheoryEquivariant} computed the stable Adams operations for $\KU_{C_2}$, and even for more general $\KU_G$.
This computation and the above map presumably could be used to recover the next result, but we will use the same techniques as for \cref{thm:TMF:main-thm}.

\begin{thm}\label{thm:KR:main-thm}
    Let $k\in \Z$ and $x \in \pi_{a+b\sigma}^{C_2} \KR[\frac{1}{k}]$.
    Then $\psi^k(x) =k^{\frac{a + b}{2}}x \in \pi_{a+b\sigma}^{C_2} \KR[\frac{1}{k}]$ if $a+b$ is even, and $\psi^k(x) =x \in \pi_{a+b\sigma}^{C_2} \KR[\frac{1}{k}]$ if $a+b$ is odd.
    In particular, if $x$ is torsion (necessarily $2$-torsion), then $\psi^k(x)=x$.
\end{thm}
\begin{proof}
    We will implicitly invert $k$ everywhere in the proof.
    As stable Adams operations are multiplicative and additive it suffices to show the statement for the generators
    $a_\sigma$, $\overline{u}$, $u^2_{2\sigma}$ and $v_0(m)$ of \ref{eq:generators:KR}.
    
    Recall that $a_\sigma \in \pi_{-\sigma}^{C_2} \KR$ is the image of $a_\sigma \in \pi_{-\sigma}^{C_2} \S$ under the unit
    $\S \to \KR$ and that the stable Adams operations $\psi^k$ are $\S$-linear with $\psi^k(1) = 1$, meaning they commute with the unit: \[\psi^k(a_\sigma) = a_\sigma\psi^k(1) = a_\sigma.\]

    The $C_2$-spectrum $\KR$ is strongly even by \cite[Examples 4.13]{GreenleesMeier:Gorenstein}, i.e., the morphism
    $\text{res}^{C_2}_e \colon \pi^{C_2}_{n(1+\sigma)} \KR  \to \pi^e_{2n}\KR =  \pi_{2n} \KU$ is an isomorphism for all $n \in \Z$.
    Since restriction maps commute with stable Adams operations by naturality, the result for $\overline{u}$ follows by the classical computation in $\KU$, cf.\ \cite[Corollary 5.2]{Adams:VectorFieldsonSpheres}.
    
    Lastly, we compute the Adams operations for the classes $u^2_{2\sigma}$ and $v_0(n)$.
    For this we use the well known equivalence $\KR^{hC_2} \simeq \KO$, see \cite{Atiyah:KtheoryReality}.
    By \cref{figure:preleiminaries:homotopyKR}, we have isomorphisms
    \begin{equation} \label{eq:computations:KR:HFP}
        \overline{u}^n \colon \pi_{n(1-\sigma)}^{C_2} \KR \xrightarrow{\simeq} \pi_{2n}^{C_2}\KR = \pi_{2n} \KO
    \end{equation} 
    given by multiplication.
    By \cite[Corollary 5.2]{Adams:VectorFieldsonSpheres}, the $\psi^k$-action on $\pi_\ast \KO$ is given by multiplication by $k^n$ in degree $2n$.
    Now compute \[k^4 \overline{u}^4 u^2_{2\sigma} = \psi^k(\overline{u}^4 u^2_{2\sigma}) = \psi^k(\overline{u}^4) \psi^k(u^2_{2\sigma}) = k^4 \overline{u}^4 \cdot \psi^k(u^2_{2\sigma}) \in \pi^{C_2}_8 \KR \simeq \Z .\]
    Since $k$ is invertible and \cref{eq:computations:KR:HFP} is an isomorphism, this implies that $u_{2\sigma}^2 = \psi^k(u_{2\sigma}^2)$.
    The same argument applied to $\overline{u}^{2m} \cdot v_0(m)$ also shows that $\psi^k(v_0(m)) = v_0(m)$, as desired.
\end{proof}

\subsection*{Topological modular forms of level structure $\Gamma_1(3)$}
The strategy to compute the stable Adams operations on $\pi^{C_2}_\bigstar \CTMF_1(3)$ is mostly the same as for Real $K$-theory.
We first need two preliminary results.

\begin{lem}\label{lem:TMF:injectivity}
    The following morphisms given by multiplication with a power of $\overline{a}_1$ are injective:
    \begin{enumerate}
        \item[(1)] $\overline{a}_1^2 \cdot \colon \pi_{2-2\sigma}^{C_2} \CTMF_1(3) \to \pi_{4}^{C_2}\CTMF_1(3) \cong \pi_{4} \TMF_0(3)$
        \item[(2)] $ \overline{a}_1^3 \cdot \colon \pi_{5-3\sigma}^{C_2} \CTMF_1(3) \to \pi_{8}^{C_2}\CTMF_1(3) \cong \pi_{8} \TMF_0(3)$ 
        \item[(3)] $ \overline{a}_1^4 \cdot\colon \pi_{4-4\sigma}^{C_2} \CTMF_1(3) \to \pi_{8}^{C_2}\CTMF_1(3) \cong \pi_{8} \TMF_0(3)$
        \item[(4)] $ \overline{a}_1^6 \cdot\colon \pi_{6-6\sigma}^{C_2} \CTMF_1(3) \to \pi_{12}^{C_2}\CTMF_1(3) \cong \pi_{12} \TMF_0(3)$
        \item[(5)] $ \overline{a}_1^8 \cdot\colon \pi_{8-8\sigma}^{C_2} \CTMF_1(3) \to \pi_{16}^{C_2}\CTMF_1(3) \cong \pi_{16} \TMF_0(3)$.
    \end{enumerate}
\end{lem}
\begin{proof}
    The equivalences in the statement are all consequences of the finite étale morphism $\topos M_1(3) \to \topos M_0(3)$, see \cite[§2]{MahowaldRezk:TMFofLevel3}.
    Injectivity of the morphisms might be checked before inverting $\overline{\Delta}$.
    Moreover, injectivity of multiplying with powers of $\overline{a}_1$ follows if source and target are torsion free.
    This is immediate by inspection of \cref{figure:preleiminaries:homotopyTMF13} for the sources and \cite[§4]{MahowaldRezk:TMFofLevel3} for the targets, see also \cite[Figure 2]{HillMeier:C2Tmf13}.
\end{proof}

\begin{lem}\label{lem:TMF:AO-on-0-3}
    Let $k \in \Z$ and $n \in \{4,8,12,16\}$.
    Then for $x \in \pi_n \TMF_0(3)[\frac{1}{k}]$ the stable Adams operations $\psi^k$ act by multiplication with $k^{n/2}$.
\end{lem}
\begin{proof}
    For any of the specified $n$, $\psi^k$ acts by multiplication with $k^{n/2}$ on $x \in \pi_n \TMF_1(3)[\frac{1}{k}]$ by \cite[Proposition 6.18]{Davies:OnLuriesTheorem}.
    Hence, $\psi^k$ acts by multiplication with $k^{n/2}$ on $E_2^{n,0}$ of the $C_2$-homotopy fixed points spectral sequence of $\TMF_1(3)$.
    In particular, if $x \in \pi_n \TMF_0(3)[\frac{1}{k}]$ is in filtration zero, then $\psi^k(x) = k^{n/2}x +y$, where $y$ is in higher filtration.
    By careful inspection of \cite[§4]{MahowaldRezk:TMFofLevel3} (or \cite[Figure 2]{HillMeier:C2Tmf13}) one learns that in these specific degrees there is nothing in higher filtration.
\end{proof}

\begin{thm}\label{thm:TMF:main-thm}
    Let $k\in \Z$ and $x \in \pi_{a+b\sigma}^{C_2} \CTMF_1(3)[\frac{1}{k}]$.
    Then $\psi^k(x) =k^{\frac{a + b}{2}}x \in \pi_{a+b\sigma}^{C_2} \CTMF_1(3)[\frac{1}{k}]$ if $a+b$ is even, and $\psi^k(x) =x \in \pi_{a+b\sigma}^{C_2} \CTMF_1(3)[\frac{1}{k}]$ if $a+b$ is odd.
    In particular, if $x$ is torsion (necessarily $2$-torsion), then $\psi^k(x)=x$.
\end{thm}
\begin{proof}
    Again, $k$ will be implicitly inverted throughout the proof.
    As stable Adams operations are multiplicative it suffices to show the statement for the generators
    $a_\sigma$, $\overline{a}_1$, $\overline{a}_3$, $u^4_{2\sigma}$, $v_0(m)$ and $\overline{a}_1(1)$ of \cite[Theorem 4.15]{HillMeier:C2Tmf13}.

    Recall that $a_\sigma \in \pi_{-\sigma}^{C_2} \CTMF_1(3)$ is the image of $a_\sigma \in \pi_{-\sigma}^{C_2} \S$ under the unit
    $\S \to \CTMF_1(3)$ and that the stable Adams operations $\psi^k$ are $\S$-linear with $\psi^k(1) = 1$, meaning they commute with the unit: \[\psi^k(a_\sigma) = a_\sigma\psi^k(1) = a_\sigma.\]
    
    By \cite[Before Proposition 4.23]{HillMeier:C2Tmf13}, $\CTMF_1(3)$ is strongly even, i.e., for all $n \in \Z$ the morphism 
    \[\text{res}^{C_2}_e \colon \pi^{C_2}_{n(1+\sigma)}\CTMF_1(3) \to
        \pi^{e}_{n(1+\sigma)}\CTMF_1(3) \cong \pi^{e}_{2n}\CTMF_1(3) \cong \pi_{2n} \TMF_1(3)\]
    is an isomorphism.
    Combining \cite[Corollary 3.3]{MahowaldRezk:TMFofLevel3} and  \cite[Theorem B]{Davies:Hecke} yields
    the result for $\overline{a}_1$ and $\overline{a}_3$.

    Now we compute $\psi^k$ for the class $u^4_{2\sigma} \in \pi^{C_2}_{8-8\sigma} \CTMF_1(3)$.
    By \cref{lem:TMF:injectivity}(5), $\overline{a}_1^8 u^4_{2\sigma}$ (rather its image) is a nonzero element in $\pi_{16} \TMF_0(3)$, where $\psi^k$ acts by multiplication with $k^{8}$ by \cref{lem:TMF:AO-on-0-3}.
    Then the calculation for $\overline{a}_1$ yields that $\psi^k(u_{2\sigma}^4)=u_{2\sigma}^4$.

    For $\overline{a}_1(1) \in \pi_{5-3\sigma}^{C_2} \CTMF_1(3)$, using \cref{lem:TMF:injectivity}(2) and \cref{lem:TMF:AO-on-0-3} gives $\psi^k(\overline{a}_1(1))= \overline{a}_1(1)$.

    Lastly, for the family of classes $v_0(m)$, it suffices to compute the effect for $m= 1,2,3$ and the others will follow by multiplicativity.
    These special cases are calculated in the same way as the two classes above using \cref{lem:TMF:AO-on-0-3} and \cref{lem:TMF:injectivity}(1),(3) and (4),respectively. 
\end{proof}
\begin{rmk}
    Note that the computation of $\psi^k(\overline{a}_1(1))$ can also be done by applying the same argument as above to $\overline{a}_3\cdot \overline{a}_1(1)$.
\end{rmk}

\bibliographystyle{amsalpha}
\bibliography{AdamsOpsTMF.bib}

\end{document}

%% file: preamble.tex
\usepackage[utf8]{inputenc}
\usepackage[english]{babel}

\usepackage{amsmath, amsthm, amssymb, amsfonts}
\usepackage{mathtools}
\usepackage{dsfont} 
\usepackage{stmaryrd}
\usepackage{tikz-cd}
\usepackage{quiver} 
    
\usepackage{hyperref}
\usepackage[capitalize,noabbrev]{cleveref}

\usepackage{enumitem}

\usepackage{comment}
\theoremstyle{plain}
\newtheorem{thm}{Theorem}[section]

\newtheorem{lem}[thm]{Lemma}

\theoremstyle{definition}

\theoremstyle{remark}
\newtheorem{rmk}[thm]{Remark}

\theoremstyle{plain}
\newtheorem{introthm}{Theorem}

\crefname{lem}{Lemma}{Lemmas}
\crefname{defn}{Definition}{Definitions}
\crefname{cor}{Corollary}{Corollaries}
\crefname{prop}{Proposition}{Propositions}
\crefname{thm}{Theorem}{Theorems}
\crefname{rmk}{Remark}{Remarks}
\crefname{section}{Section}{Sections}
\crefname{equation}{Equation}{Equivalences}
\crefname{conj}{Conjecture}{Conjectures}

\newcommand{\Z}{{\mathbb{Z}}}

\newcommand{\E}{{\mathbb{E}}}
\newcommand{\G}{{\mathbb{G}}}
\renewcommand{\S}{{\mathbb S}}
\newcommand{\finfld}[1]{\mathbb{F}_{#1}}

\renewcommand{\lim}[1]{{\underset{\substack{#1}}{\operatorname{lim}}\,}}

\newcommand{\Sp}{{\operatorname{Sp}}}

\newcommand{\op}[1]{#1^{\operatorname{op}}}

\newcommand{\topos}[1]{{\mathcal {#1}}}
\newcommand{\Cat}[1]{{\mathcal {#1}}}

\newcommand{\CAlg}[1]{{\operatorname{CAlg}(#1)}}

\newcommand{\id}[1]{\operatorname{id}_{#1}}

\newcommand{\TMF}{\operatorname{TMF}}

\newcommand{\CTMF}{\mathbf{TMF}}

\newcommand{\Ctmf}{\mathbf{tmf}}

\newcommand{\KU}{\operatorname{KU}}

\newcommand{\KO}{\operatorname{KO}}

\newcommand{\KR}{\mathbf{KR}}
\newcommand{\kR}{\mathbf{kR}}

\newcommand{\Me}{\topos{M}_{ell}} 
\newcommand{\MT}{\topos{M}_1(3)} 
\newcommand{\Isog}{{\operatorname{Isog}}}
\newcommand{\Ot}{\Cat{O}^{\textup{top}}} 

\author{Anton Engelmann\footnote{\href{https://anton-engelmann.com}{www.anton-engelmann.com}}}
\date{}

%% file: AdamsOpsTMF.bib
@article{AdamsAtiyah:HopfOne,
	author = {Adams, J. F. and Atiyah, M. F.},
    title = {{$K$}-theory and the {H}opf invariant},
    journal = {Quart. J. Math. Oxford Ser. (2)},
    pages = {31--38},
    volume = {17},
	year = {1966},
	fjournal = {The Quarterly Journal of Mathematics. Oxford. Second Series}
	}

@article{Adams:VectorFieldsonSpheres,
	author = {J. F. Adams},
	issn = {0003486X, 19398980},
	journal = {Annals of Mathematics},
	number = {3},
	pages = {603--632},
	publisher = {[Annals of Mathematics, Trustees of Princeton University on Behalf of the Annals of Mathematics, Mathematics Department, Princeton University]},
	title = {Vector Fields on Spheres},
	url = {http://www.jstor.org/stable/1970213},
	urldate = {2025-08-26},
	volume = {75},
	year = {1962}}

@article{Atiyah:KtheoryReality,
	author = {Atiyah, M. F.},
	doi = {10.1093/qmath/17.1.367},
	fjournal = {The Quarterly Journal of Mathematics. Oxford. Second Series},
	issn = {0033-5606,1464-3847},
	journal = {Quart. J. Math. Oxford Ser. (2)},
	pages = {367--386},
	title = {{$K$}-theory and reality},
	url = {https://doi.org/10.1093/qmath/17.1.367},
	volume = {17},
	year = {1966}}

@article{Balderrama:KTheoryEquivariant,
	author = {Balderrama, William},
	issn = {1076-9803},
	journal = {New York J. Math.},
	pages = {1531--1553},
	title = {{$K$}-theory equivariant with respect to an elementary abelian 2-group},
	url = {http://nyjm.albany.edu/j/2022/28-67.html},
	volume = {28},
	year = {2022}}

@book{BrunerGreenlees:ConnectiveRealK,
	author = {Bruner, Robert R. and Greenlees, J. P. C.},
	doi = {10.1090/surv/169},
	isbn = {978-0-8218-5189-0},
	pages = {vi+318},
	publisher = {American Mathematical Society, Providence, RI},
	series = {Mathematical Surveys and Monographs},
	title = {Connective real {$K$}-theory of finite groups},
	url = {https://doi.org/10.1090/surv/169},
	volume = {169},
	year = {2010}}

@article{Davies:ConstructingCalculatingAO,
	author = {Davies, Jack Morgan},
	doi = {10.4171/dm/1005},
	fjournal = {Documenta Mathematica},
	issn = {1431-0635,1431-0643},
	journal = {Doc. Math.},
	number = {4},
	pages = {935--979},
	title = {Constructing and calculating {A}dams operations on dualisable topological modular forms},
	url = {https://doi.org/10.4171/dm/1005},
	volume = {30},
	year = {2025}}

@article{Davies:Hecke,
	author = {Jack Morgan Davies},
	doi = {https://doi.org/10.1016/j.aim.2024.109828},
	issn = {0001-8708},
	journal = {Advances in Mathematics},
	pages = {109828},
	title = {Hecke operators on topological modular forms},
	url = {https://www.sciencedirect.com/science/article/pii/S0001870824003438},
	volume = {452},
	year = {2024}}

@article{Davies:OnLuriesTheorem,
	author = {Davies, Jack Morgan},
	doi = {10.1007/s00209-024-03626-1},
	fjournal = {Mathematische Zeitschrift},
	issn = {0025-5874,1432-1823},
	journal = {Math. Z.},
	number = {2},
	pages = {Paper No. 24, 53},
	title = {On {L}urie's theorem and applications},
	url = {https://doi.org/10.1007/s00209-024-03626-1},
	volume = {309},
	year = {2025}}

@misc{DaviesLinskens:Global-Tate-K,
	archiveprefix = {arXiv},
	author = {Jack Morgan Davies and Sil Linskens},
	eprint = {2503.04494},
	primaryclass = {math.AT},
	title = {On the derived Tate curve and global smooth Tate $K$-theory},
	url = {https://arxiv.org/abs/2503.04494},
	year = {2025},
    note = {\url{https://arxiv.org/abs/2503.04494}}}

@article{Dugger:AtHiSpSeqforKR,
	author = {Dugger, Daniel},
	doi = {10.1007/s10977-005-1552-9},
	fjournal = {$K$-Theory. An Interdisciplinary Journal for the Development, Application, and Influence of $K$-Theory in the Mathematical Sciences},
	issn = {0920-3036,1573-0514},
	journal = {$K$-Theory},
	number = {3-4},
	pages = {213--256},
	title = {An {A}tiyah-{H}irzebruch spectral sequence for {$KR$}-theory},
	url = {https://doi.org/10.1007/s10977-005-1552-9},
	volume = {35},
	year = {2005}}

@incollection{Greenlees:4reality,
	author = {Greenlees, J. P. C.},
	booktitle = {An alpine bouquet of algebraic topology},
	doi = {10.1090/conm/708/14261},
	isbn = {978-1-4704-2911-9},
	pages = {139--156},
	publisher = {Amer. Math. Soc., [Providence], RI},
	series = {Contemp. Math.},
	title = {Four approaches to cohomology theories with reality},
	url = {https://doi.org/10.1090/conm/708/14261},
	volume = {708},
	year = {[2018] \copyright 2018}}

@article{GreenleesMeier:Gorenstein,
	author = {Greenlees, J. P. C. and Meier, Lennart},
	doi = {10.2140/agt.2017.17.3547},
	fjournal = {Algebraic \& Geometric Topology},
	issn = {1472-2747,1472-2739},
	journal = {Algebr. Geom. Topol.},
	number = {6},
	pages = {3547--3619},
	title = {Gorenstein duality for real spectra},
	url = {https://doi.org/10.2140/agt.2017.17.3547},
	volume = {17},
	year = {2017}}

@article{HillHopkinsRavenel:KervaireInvariantNonExistence,
	author = {Hill, M. A. and Hopkins, M. J. and Ravenel, D. C.},
	doi = {10.4007/annals.2016.184.1.1},
	fjournal = {Annals of Mathematics. Second Series},
	issn = {0003-486X,1939-8980},
	journal = {Ann. of Math. (2)},
	number = {1},
	pages = {1--262},
	title = {On the nonexistence of elements of {K}ervaire invariant one},
	url = {https://doi.org/10.4007/annals.2016.184.1.1},
	volume = {184},
	year = {2016}}

@article{HillMeier:C2Tmf13,
	author = {Hill, Michael A. and Meier, Lennart},
	doi = {10.2140/agt.2017.17.1953},
	fjournal = {Algebraic \& Geometric Topology},
	issn = {1472-2747,1472-2739},
	journal = {Algebr. Geom. Topol.},
	number = {4},
	pages = {1953--2011},
	title = {The {$C_2$}-spectrum {${\rm Tmf}_1(3)$} and its invertible modules},
	url = {https://doi.org/10.2140/agt.2017.17.1953},
	volume = {17},
	year = {2017}}

@inproceedings{Hopkins:TopologicalModularForms,
	address = {Basel},
	author = {Hopkins, Michael J.},
	booktitle = {Proceedings of the International Congress of Mathematicians},
	editor = {Chatterji, S. D.},
	isbn = {978-3-0348-9078-6},
	pages = {554--565},
	publisher = {Birkh{\"a}user Basel},
	title = {Topological Modular Forms, the Witten Genus, and the Theorem of the Cube},
	year = {1995}}

@misc{EllipticII,
	author = {Jacob Lurie},
	month = {April},
	title = {Elliptic Cohomology II: Orientations},
	note = {Available online at \url{https://www.math.ias.edu/~lurie/papers/Elliptic-II.pdf}},
	year = {2018}}

@article{MahowaldRezk:TMFofLevel3,
	author = {Mahowald, Mark and Rezk, Charles},
	doi = {10.4310/PAMQ.2009.v5.n2.a9},
	fjournal = {Pure and Applied Mathematics Quarterly},
	issn = {1558-8599,1558-8602},
	journal = {Pure Appl. Math. Q.},
	number = {2},
	pages = {853--872},
	title = {Topological modular forms of level 3},
	url = {https://doi.org/10.4310/PAMQ.2009.v5.n2.a9},
	volume = {5},
	year = {2009}}

@article{Quillen:CohomologyKTheoryGL,
	author = {Quillen, Daniel},
	doi = {10.2307/1970825},
	fjournal = {Annals of Mathematics. Second Series},
	issn = {0003-486X},
	journal = {Ann. of Math. (2)},
	pages = {552--586},
	title = {On the cohomology and {$K$}-theory of the general linear groups over a finite field},
	url = {https://doi.org/10.2307/1970825},
	volume = {96},
	year = {1972}}

@article{Segal:EquivariantKTheory,
	author = {Segal, Graeme},
	fjournal = {Institut des Hautes \'Etudes Scientifiques. Publications Math\'ematiques},
	issn = {0073-8301,1618-1913},
	journal = {Inst. Hautes \'Etudes Sci. Publ. Math.},
	number = {34},
	pages = {129--151},
	title = {Equivariant {$K$}-theory},
	url = {http://www.numdam.org/item?id=PMIHES_1968__34__129_0},
	year = {1968}}
